\documentclass[11pt]{amsproc}
\usepackage{color}
\newtheorem{theorem}{\bf Theorem}[section]
\newtheorem{lemma}[theorem]{\bf Lemma}
\newtheorem{proposition}[theorem]{\bf Proposition}
\newtheorem{corollary}[theorem]{\bf Corollary}

\newtheorem{definition}[theorem]{\bf Definition}

\author[C. Acciarri]{Cristina Acciarri}
\address{Department of Mathematics, University of Brasilia, 70910-900 Bras\'ilia DF, Brazil}
\email{acciarricristina@yahoo.it}

\author[D.\,S. da Silveira ]{Danilo San\c c\~ao da Silveira}
\address{Department of Mathematics, Federal University of Goias, 75704-020 Catal\~ao GO, Brazil}
\email{sancaodanilo@ufg.br}


\thanks{This work was supported by the Conselho Nacional de Desenvolvimento Cient\'{\i}fico e Tecnol\'ogico (CNPq), Brazil. }

\title[Profinite groups and centralizers of coprime automorphisms]{ Profinite groups and centralizers of coprime automorphisms whose elements are Engel}

\begin{document}

\begin{abstract}  Let $q$ be a prime, $n$ a positive integer and $A$ an elementary abelian group of order $q^r$ with $r\geq2$ acting on a finite $q'$-group $G$.  The following results are proved.

We show that if all elements in $\gamma_{r-1}(C_G(a))$ are $n$-Engel in $G$ for any $a\in A^\#$, then $\gamma_{r-1}(G)$ is $k$-Engel for some $\{n,q,r\}$-bounded number $k$,  and  if, for some integer $d$ such that $2^d\leq r-1$, all elements in the $d$th derived group of $C_G(a)$ are $n$-Engel in $G$ for any $a\in A^\#$, then the $d$th derived group $G^{(d)}$ is $k$-Engel for some $\{n,q,r\}$-bounded number $k$.

Assuming $r\geq 3$ we prove that if all elements in $\gamma_{r-2}(C_G(a))$ are $n$-Engel in $C_G(a)$ for any $a\in A^\#$, then $\gamma_{r-2}(G)$ is $k$-Engel for some $\{n,q,r\}$-bounded number $k$, and  if, for some integer $d$ such that $2^d\leq r-2$, all elements in the $d$th derived group of $C_G(a)$ are $n$-Engel in $C_G(a)$ for any $a\in A^\#,$ then the $d$th derived group $G^{(d)}$ is $k$-Engel for some $\{n,q,r\}$-bounded number $k$.

Analogue (non-quantitative) results for profinite groups are also obtained.
\end{abstract}

\maketitle
\section{Introduction}
Let $A$ be a  finite group acting on a finite group $G$. Many well-known results show that the structure of the centralizer $C_G(A)$ (the fixed-point subgroup) of $A$ has influence over the structure of $G$. The influence is especially strong if $(|A|,|G|)=1$, that is, the action of $A$ on $G$ is coprime. Following the solution of the restricted Burnside problem it was discovered that the exponent of $C_G(A)$ may have strong impact over the exponent of $G$. Remind that a group $G$ is said to have exponent $n$ if $x^n=1$ for every $x\in G$ and $n$ is the minimal positive integer with this property. The next theorem was obtained in \cite{KS}.

\begin{theorem}\label{q2}
Let $q$ be a prime, $n$ a positive integer and $A$ an elementary abelian group of order $q^2$. Suppose that $A$ acts coprimely on a finite group $G$ and assume that $C_{G}(a)$ has exponent dividing $n$ for each $a\in A^{\#}$. Then the exponent of $G$ is $\{n,q\}$-bounded.
\end{theorem}  
Here and throughout the paper $A^{\#}$ denotes the set of nontrivial elements of $A$. Moreover we use the expression ``$\{a,b,\dots\}$-bounded'' to abbreviate ``bounded from above in terms of  $a,b,\dots$ only''. The proof of  the above result involves a number of deep ideas. In
particular, Zelmanov's techniques that led to the solution of the restricted Burnside problem \cite{Z1,Z0} are combined with the Lubotzky--Mann theory of powerful $p$-groups \cite{luma}, 
and a theorem of Bahturin and Zaicev on Lie algebras admitting a group of automorphisms whose fixed-point subalgebra is PI \cite{BZ}.

A profinite (non-quantitative) version of the above theorem was established in \cite{eu}. In the context of profinite groups all the usual concepts of groups theory are interpreted topologically. In particular, by a subgroup of a profinite group we always mean a closed subgroup and   a subgroup is said to be generated by a set $S$ if it is topologically generated by $S$. By an automorphism of a profinite group we mean a continuous automorphism.  We say that a group $A$ acts on a profinite group $G$ \emph{coprimely} if  $A$ is finite while $G$ is an inverse limit of finite groups whose orders are relatively prime to the order of $A$.   
The profinite (non-quantitative) version of Theorem \ref{q2} is as follows.

\begin{theorem}
Let $q$ be a prime and $A$ an elementary abelian group of order  $q^2$. Suppose that $A$ acts coprimely on a profinite group $G$ and assume that $C_G(a)$ is torsion for each $a\in A^{\#}$. Then $G$ is locally finite.
\end{theorem}

In \cite{shusa}  the situation where the centralizers $C_{G}(a)$ consist of Engel elements was dealt with. If  $x,y$ are elements of a (possibly infinite) group $G$, the commutators $[x,_n y]$ are defined inductively by the rule
$$[x,_0 y]=x,\quad [x,_n y]=[[x,_{n-1} y],y]\quad \text{for all}\, n\geq 1.$$
An element $x$ is called a (left) Engel element if for any $g\in G$ there exists $n$, depending on $x$ and $g$, such that $[g,_n x]=1$.  A group $G$ is called Engel if all elements of $G$ are Engel. The element $x$ is called a (left) $n$-Engel element if for any $g\in G$ we have $[g,_n x]=1$. The group $G$ is $n$-Engel if all elements of $G$ are $n$-Engel. The following result was proved in \cite{shusa}.

\begin{theorem}\label{a1} 
Let $q$ be a prime, $n$ a positive integer and $A$ an elementary abelian group of order $q^2$. Suppose that $A$ acts coprimely on a finite group $G$ and assume that for each $a\in A^{\#}$ every element of $C_{G}(a)$ is $n$-Engel in $G$. Then the group $G$ is $k$-Engel for some $\{n,q\}$-bounded number $k$.
\end{theorem} 
A profinite (non-quantitative) version of the above theorem was established in the recent work \cite{CPD}.

\begin{theorem}\label{b1}  
Let $q$ be a prime and $A$ an elementary abelian group of order  $q^2$. Suppose that $A$ acts coprimely on a profinite group $G$ and assume that all elements in $C_{G}(a)$ are Engel in $G$ for each $a\in A^{\#}$. Then $G$ is locally nilpotent.
\end{theorem}

A very deep theorem of Wilson and Zelmanov \cite[Theorem 5]{WZ} tells us that a profinite group is locally nilpotent if and only if it is Engel. Thus, there is a clear relation between Theorem \ref{a1} and Theorem \ref{b1}.

If, in Theorem \ref{a1}, we relax the hypothesis that every element of $C_{G}(a)$ is $n$-Engel in $G$ and require instead that every element of $C_{G}(a)$ is $n$-Engel in $C_{G}(a)$, we quickly see that the result is no longer true. An example of a finite non-nilpotent group $G$ admitting  an action of a noncyclic group $A$ of
order four  such that $C_G(a)$ is abelian for each $a\in A^{\#}$ can be found for instance in \cite{PS-CA3}. On the other hand, another result, that was established in \cite{CPD}, is the following.

\begin{theorem}\label{a2} Let $q$ be a prime, $n$ a positive integer and $A$ an elementary abelian group of order $q^3$. Suppose that $A$ acts coprimely on a finite group $G$ and assume that for each $a\in A^{\#}$ every element of $C_{G}(a)$ is $n$-Engel in $C_{G}(a)$. Then the group $G$ is $k$-Engel for some $\{n,q\}$-bounded number $k$.
\end{theorem}

In \cite{PS-CA3} a profinite (non-quantitative) version of Theorem \ref{a2} was obtained. The statement is as follows.

 \begin{theorem}\label{b2} Let $q$ be a prime and $A$ an elementary abelian group of order  $q^3$. Suppose that $A$ acts coprimely on a profinite group $G$ and assume that $C_{G}(a)$ is locally nilpotent for each $a\in A^{\#}$. Then $G$ is locally nilpotent.
\end{theorem}

The relation, noted above between  Theorem \ref{a1} and Theorem \ref{b1}, naturally extends to  Theorems \ref{a2} and \ref{b2}.

Let us denote by $\gamma_i(H)$ the $i$th term of the lower central series of a group $H$ and  by $H^{(i)}$ the $i$th term of the derived series of $H$. It was shown in \cite{PS-Gu} and further in \cite{PS-CA, PS-CA2} that if the rank of the acting group $A$ is big enough, then results of similar nature to that  of Theorem \ref{q2} can be obtained while imposing conditions on elements of $\gamma_i(C_G(a))$ or ${C_G(a)}^{(i)}$ rather than  on elements of $C_G(a)$. In the same spirit, one of the goals of the present article is to extend  Theorems \ref{a1} and \ref{a2}  respectively  as follows. 

\begin{theorem}\label{aa1} Let $q$ be a prime, $n$ a positive integer and $A$ an elementary abelian group of order $q^r$ with $r\geq2$ acting on a finite $q'$-group $G$.

$(1)$ If all elements in $\gamma_{r-1}(C_G(a))$ are $n$-Engel in $G$ for any $a\in A^\#$, then $\gamma_{r-1}(G)$ is $k$-Engel for some $\{n,q,r\}$-bounded number $k$.

$(2)$ If, for some integer $d$ such that $2^d\leq r-1$, all elements in the $d$th derived group of $C_G(a)$ are $n$-Engel in $G$ for any $a\in A^\#$, then the $d$th derived group $G^{(d)}$ is $k$-Engel for some $\{n,q,r\}$-bounded number $k$.
\end{theorem}

\begin{theorem}\label{aa2} Let $q$ be a prime, $n$ a positive integer and $A$ an elementary abelian group of order $q^r$ with $r\geq3$ acting on a finite $q'$-group $G$.

$(1)$ If all elements in $\gamma_{r-2}(C_G(a))$ are $n$-Engel in $C_G(a)$ for any $a\in A^\#$, then $\gamma_{r-2}(G)$ is $k$-Engel for some $\{n,q,r\}$-bounded number $k$.

$(2)$ If, for some integer $d$ such that $2^d\leq r-2$, all elements in the $d$th derived group of $C_G(a)$ are $n$-Engel in $C_G(a)$ for any $a\in A^\#,$ then the $d$th derived group $G^{(d)}$ is $k$-Engel for some $\{n,q,r\}$-bounded number $k$.
\end{theorem}

Let $H,K$ be subgroups  of a profinite group $G$. We denote by $[H,K]$ the closed  subgroup of $G$ generated by all commutators of the form $[h,k]$, with $h\in H$ and $k\in K$.  Thus one can consider inductively the  following closed  subgroups:
$$\gamma_1(G)=G,\quad \gamma_k(G)=[\gamma_{k-1}(G),G], \,\,\text{for}\, k\geq 1\ \mbox{and}$$
$$G^{(0)}=G,\quad G^{(k)}=[G^{(k-1)},G^{(k-1)}],\,\text{for}\, k\geq 1.$$
Finally, we formulate the (non-quantitative) analogues of Theorems \ref{aa1} and \ref{aa2}, respectively.

\begin{theorem}\label{bb1} Let $q$ be a prime, $n$ a positive integer and $A$ an elementary abelian group of order $q^r$ with $r\geq2$ acting coprimely on a profinite group $G$.

$(1)$ If all elements in $\gamma_{r-1}(C_G(a))$ are Engel in $G$ for any $a\in A^\#$, then $\gamma_{r-1}(G)$ is locally nilpotent.

$(2)$ If, for some integer $d$ such that $2^d\leq r-1$, all elements in the $d$th derived group of $C_G(a)$ are Engel in $G$ for any $a\in A^\#$, then the $d$th derived group $G^{(d)}$ is locally nilpotent.
\end{theorem}

\begin{theorem}\label{bb2} Let $q$ be a prime, $n$ a positive integer and $A$ an elementary abelian group of order $q^r$ with $r\geq3$ acting coprimely on a profinite group $G$.

$(1)$ If all elements in $\gamma_{r-2}(C_G(a))$ are Engel in $C_G(a)$ for any $a\in A^\#$, then $\gamma_{r-2}(G)$ is locally nilpotent.

$(2)$ If, for some integer $d$ such that $2^d\leq r-2$, all elements in the $d$th derived group of $C_G(a)$ are Engel in $C_G(a)$ for any $a\in A^\#$, then the $d$th derived group $G^{(d)}$ is locally nilpotent.
\end{theorem}

Thus the purpose of the present article is to provide the proofs for Theorems \ref{aa1}, \ref{aa2}, \ref{bb1} and \ref{bb2}. The paper is organized as follows.  In  Sections 2 and 3 we present the Lie-theoretic machinery that will be  useful within  the proofs. Later in Section 4 a technical tool, introduced in \cite{PS-CA}, is extended to the context of profinite groups. Sections 5 and 6  are devoted to proving Theorems \ref{aa1} and \ref{aa2}. Finally in Section 7 we give the details of the proofs of Theorems \ref{bb1} and \ref{bb2}. 

Throughout the paper we use, without special references, the well-known properties of coprime actions (see for example \cite[Lemma 3.2]{PSprofinite}):
 
 If $\alpha$ is a coprime automorphism of a profinite group $G$, then \linebreak[4] $C_{G/N}(\alpha)=C_G(\alpha)N/N$ for any $\alpha$-invariant normal subgroup $N$.
 
If $A$ is a noncyclic abelian group acting coprimely on a profinite group $G$, then $G$ is generated by the subgroups $C_G(B)$, where $A/B$ is cyclic.

        \section{Results on  Lie algebras and Lie rings}
Let $X$ be a subset of a Lie algebra $L$. By a commutator in elements of $X$ we mean any element of $L$ that can be obtained as a Lie product of elements of $X$ with some system of brackets. If $x_1,\ldots,x_k,x, y$ are elements of $L$, we define inductively 
$$[x_1]=x_1; [x_1,\ldots,x_k]=[[x_1,\ldots,x_{k-1}],x_k]$$
and 
$[x,_0y]=x; [x,_my]=[[x,_{m-1}y],y],$ for all positve integers $k,m$.  
As usual, we say that an element $a\in L$ is ad-nilpotent if there exists a positive integer $n$ such that $[x,_na]=0$ for all $x\in L$. If $n$ is the least integer with the above property, then we say that $a$ is ad-nilpotent of index $n$.

The next theorem represents the most general form of the Lie-theoretical part of the solution of the restricted Burnside problem. It was announced by Zelmanov in \cite{Z1}. A detailed proof was  published in \cite{zenew}.

\begin{theorem}\label{Z1992}
Let $L$ be a Lie algebra over a field and suppose that $L$ satisfies a polynomial identity. If $L$ can be generated by a finite set $X$ such that every commutator in elements of $X$ is ad-nilpotent, then $L$ is nilpotent.
\end{theorem}

The next theorem, which was proved by Bahturin and Zaicev for soluble groups $A$ \cite{BZ} and later extended by Linchenko to the general case \cite{l}, provides an important criterion for a Lie algebra to satisfy a polynomial identity. 
 
\begin{theorem}\label{blz}
Let $L$ be a Lie algebra over a field $K$. Assume that a finite group $A$ acts on $L$ by automorphisms in such a manner that $C_L(A)$ satisfies a polynomial identity. Assume further that the characteristic of $K$ is either $0$ or prime to the order of $A$. Then $L$ satisfies a polynomial identity.
\end{theorem}

Both Theorems \ref{Z1992} and  \ref{blz} admit  the following respective quantitative versions (see for example \cite{KS} and  \cite{aaaa}).

\begin{theorem}\label{Z1}
Let $L$ be a Lie algebra over a field $K$ generated by $a_1,\ldots,a_m$. Suppose that $L$ satisfies a polynomial identity $f\equiv 0$ and each commutator in $a_1,\ldots,a_m$ is ad-nilpotent of index at most $n$. Then $L$ is nilpotent of $\{f,K,m,n\}$-bounded class.
\end{theorem}

\begin{theorem}\label{bZ1}
 Let $L$ be as in Theorem \ref{blz} and assume that $C_L(A)$ satisfies a polynomial identity $f\equiv 0$. Then $L$ satisfies a polynomial identity of $\{|A|,f,K\}$-bounded degree.
\end{theorem}
By combining Theorems \ref{Z1} and \ref{bZ1} the following corollary can be obtained.
\begin{corollary}\label{ZelBazaField}
Let $L$ be a Lie algebra over a field $K$ and $A$ a finite group of automorphisms of $L$ such that $C_L(A)$ satisfies the polynomial identity $f\equiv 0$. Suppose that the characterisitic of $K$ is either $0$ or prime to the order of $A$. Assume that $L$ is generated by an $A$-invariant set of $m$ elements in which every commutator is ad-nilpotent of index at most $n$. Then $L$ is nilpotent of $\{|A|,f,K,m,n\}$-bounded class.
\end{corollary}

For our purpose we will need to work with Lie rings, and not only with Lie algebras.  As usual, $\gamma_i(L)$ denotes the $i$th term of the lower central series of $L$.  In \cite{shusa} it was established the following  result for Lie rings, similar to Corollary \ref{ZelBazaField}.

\begin{theorem}\label{BazaRing}
Let $L$ be a Lie ring and $A$ a finite group of automorphisms of $L$ such that  
$C_L(A)$ satisfies the polynomial identity $f\equiv 0$. Further, assume that $L$ is generated by an $A$-invariant set of $m$ elements such that every commutator in the 
generators is ad-nilpotent of index at most $n$. Then there exist positive integers 
$e$ and $c$, depending only on $|A|, f, m$ and $n$, such that  $e\gamma_c(L)=0$.
\end{theorem}

We also require the following useful lemma whose proof can be found in \cite{KS}. 

\begin{lemma}\label{lemmanovo}
Let $L$ be a Lie ring and $H$ a subring of $L$ generated by $m$ elements $h_{1},\ldots, h_{m}$ such that all commutators in $ h_i$ are  ad-nilpotent in $L$ of index at most $n$. If $H$ is nilpotent of class $c$, then for some $\{c, m,n\}$-bounded number $u$ we have  $[L,\underbrace{H,\ldots, H}_u]=0.$
\end{lemma}

Recall that the identity $$\sum_{\sigma\in S_n}[y,x_{\sigma(1)},\ldots, x_{\sigma(n)}]\equiv 0$$ is called linearized $n$-Engel identity. In general, Theorem \ref{Z1992} cannot be extended to the case where $L$ is just a Lie ring (rather than a Lie algebra over a field). However such an extension does hold in the particular case where the polynomial identity $f\equiv 0$ is a linearized Engel identity.  More precisely, by combining Theorems \ref{BazaRing}, \ref{Z1} and Lemma \ref{lemmanovo}  the following result can be obtained. See \cite{shusa} for further details.

\begin{theorem}\label{Zelring}
Let $F$ be the free Lie ring and $f$ an element of $F$ (Lie polynomial) such that $f\not\in pF$ for any prime $p$. Suppose that $L$ is a Lie ring generated by finitely many elements $ a_1,\ldots, a_m$ such that all commutators in the generators are ad-nilpotent of index at most $n$. Assume that $L$ satisfies the identity $f\equiv 0$. Then $L$ is nilpotent with $\{f,m,n\}$-bounded class.
\end{theorem}

      \section{On associated Lie rings}
Given a group $G$, there are several well-known ways to associate a Lie ring to it (see \cite{Huppert2,Khu1,aaaa}). For the reader's convenience we will briefly describe the construction that we are using in the present paper.

A series of subgroups for $G$ $$G=G_1\geq G_2\geq\cdots\eqno{(*)}$$ is called an $N$-series if it satisfies $[G_i,G_j]\leq G_{i+j}$ for all $i,j$. Obviously any $N$-series is central, i.e. $G_i/G_{i+1}\leq Z(G/G_{i+1})$ for any $i$. Given an $N$-series $(*)$, let $L^*(G)$ be the direct sum of the abelian groups $L_i^*=G_i/G_{i+1}$, written additively. Commutation in $G$ induces a binary operation $[,]$ in $L^*(G)$. For homogeneous elements $xG_{i+1}\in L_i^*,yG_{j+1}\in L_j^*$ the operation is defined by $$[xG_{i+1},yG_{j+1}]=[x,y]G_{i+j+1}\in L_{i+j}^*$$ and extended to arbitrary elements of $L^*(G)$ by linearity. It is easy to check that the operation is well-defined and that $L^*(G)$ with the operations $+$ and $[,]$ is a Lie ring. 

An $N$-series $(*)$ is called  an $N_p$-series if $G_i^p\leq G_{pi}$ for all $i$. An important example of an $N_p$-series is  the  case where the series $(*)$ is the $p$-dimension central series, also known under the name of Zassenhaus-Jennings-Lazard series (see \cite[p.\ 250]{Huppert2} for details).
Observe that if all quotients $G_i/G_{i+1}$ of an $N$-series $(*)$ have prime exponent $p$ then $L^*(G)$ can be viewed as a Lie algebra over the field with $p$ elements.

Any automorphism of $G$ in the natural way induces an automorphism of $L^*(G)$. If $G$ is profinite and $\alpha$ is a coprime automorphism of $G$, then the subring (subalgebra) of fixed points of $\alpha$ in $L^*(G)$ is isomorphic to the Lie ring (algebra) associated to the group $C_G(\alpha)$ via the series formed by intersections of $C_G(\alpha)$ with the series $(*)$ (see \cite{aaaa} for further details).

 In the case where the series $(*)$ is just the lower central series of $G$ we write $L(G)$ for the associated Lie ring. In the case where the series $(*)$ is the $p$-dimension central series of $G$ we write $L_p(G)$ for the subalgebra generated by the first homogeneous component $G_1/G_2$ in the associated Lie algebra over the field with $p$ elements. 

Let $H$ be a subgroup of $G$. For a series $(*)$  we write 
$$L^*(G,H)=\bigoplus_{j\geq 1}\dfrac{(G_j\cap H)G_{j+1}}{G_{j+1}}$$
and, if the series $(*)$ is the $p$-dimension central series, we write
$$ L_p(G,H)=L_p(G)\cap L^*(G,H).$$
In particular, if a group $A$ acts coprimely on $G$, then we have 
$$L^*(G,C_G(A))=C_{L^*(G)}(A)\ \mbox{and}\ L_p(G,C_G(A))=C_{L_p(G)}(A).$$

We will also require the following lemma that essentially is due to Wilson and Zelmanov (cf. \cite[Lemma in Section 3]{WZ}).
\begin{lemma}\label{leWZ} Let $G$ be a profinite group and $g\in G$ an element such that for any $x\in G$ there exists a positive $n$ with the property that $[x,{}_n\,g]=1$. Let $L^*(G)$ be the Lie algebra associated with $G$ using an $N_p$-series $(*)$ for some prime $p$. Then the image of $g$ in $L^*(G)$ is ad-nilpotent.
\end{lemma}

We  close this section by quoting the following result
whose proof can be found in \cite{KS}.

\begin{theorem}\label{finiteL}
Let $P$ be a $d$-generated finte $p$-group and suppose that $L_p(G)$ is nilpotent of class $c$. Then  $P$  has a powerful characteristic subgroup of  $\{p,c,d\}$-bounded index.
\end{theorem}

Remind that powerful $p$-groups were introduced by Lubotzky and Mann in \cite{luma}.  A finite $p$-group $P$ is said to be powerful if and only if $[P,P]\leq P^{p}$ for $p \neq 2$ (or $[P,P]\leq P^{4}$ for $p=2$), where  $P^{i}$  denotes   the subgroup of $P$ generated by all $i$th powers.  Powerful  $p$-groups have some nice properties. In particular, if $P$ is a powerful $p$-group, then the subgroups $\gamma_{i}(P), P^{(i)}$ and $P^{i}$ are also powerful. Moreover, for given positive integers $n_{1},\ldots,n_{s}$, it follows, by repeated applications of \cite[Propositions 1.6 and 4.1.6]{luma}, that
$$[P^{n_{1}},\ldots,P^{n_{s}}]\leq \gamma_{s}(P)^{n_{1}\cdots n_{s}}.$$
Furthermore if a powerful $p$-group $P$ is generated by $d$ elements, then any subgroup of $P$ can be generated by at most $d$ elements and $P$ is a product of $d$ cyclic subgroups, that is, $P$ has cyclic subgroups $C_1,\dots,C_d$ with the property that for every element $x\in P$ there exist $x_1\in C_1,\dots,x_d\in C_d$ such that $x=x_1\dots x_d$.

\section{On a technical tool: $A$-special subgroups}
The main step in order to deal with the proof of  part (2) of Theorems \ref{aa1} and \ref{aa2} is to consider  the case where $G$ is a $p$-group, which can be treated via Lie methods. Then the general case will follows from a reduction to the case of $p$-groups.  We will deal with the case of $p$-groups by combining   Lie methods with the use of  the technical concept of $A$-special subgroups of  a group $G$. This concept was introduced in \cite{PS-CA}. In what follows, we are going to provide the reader with the most relevant informations on that topic.  Let us start by recalling the definition.

\begin{definition}
Let $A$ be an elementary abelian $q$-group acting on a finite $q'$-group $G$. Let $A_1,\ldots,A_s$ be the maximal subgroups of $A$ and $H$ a subgroup of $G$. 

$(1)$ We say that $H$ is an $A$-special subgroup of $G$ of degree $0$ if and only if  $H=C_G(A_i)$ for suitable $i\leq s$.

$(2)$Suppose that $k\geq 1$ and the $A$-special subgroups of $G$ of degree $k-1$ are defined. Then $H$ is an $A$-special subgroup of $G$ of degree $k$ if and only if   there exist $A$-special subgroups $J_1, J_2$ of $G$ of degree $k-1$ such that  $H=[J_1, J_2]\cap C_G(A_i)$ for suitable $i\leq s$.
\end{definition}

Note that the $A$-special subgroups of $G$ of any degree are $A$-invariant. If $A$ has order $q^r$, then for given integer $k$ the number of $A$-special subgroups of $G$ of degree $k$ is bounded in terms of $q,r$ and $k$. Moreover the $A$-special subgroups have nice properties that are crucial for our purpose. We state here some of those properties whose proofs can be found in \cite[Proposition 3.2, Theorem 4.1 and Corollary 4.4]{PS-CA}.

\begin{proposition}\label{propPS-CA2}
Let $A$ be an elementary abelian $q$-group of order $q^r$ with $r\geq 2$ acting on a finite $q'$-group $G$ and $A_1,\ldots,A_s$ be the maximal subgroups of $A$. Let $k\geq 0$ be an integer.

$(1)$ If $k\geq 1$, then every $A$-special subgroup of $G$ of degree $k$ is contained in some $A$-special subgroup of $G$ of degree $k-1$.

$(2)$ If $2^k\leq r-1$ and $H$ is an $A$-special subgroup of $G$ of degree $k$, then $H$ is contained in the $k$th derived group of $C_G(B)$ for some subgroup $B\leq A$ such that $|A/B|\leq q^{2^k}.$ 

 $(3)$ Let $R_k$ be the subgroup generated by all $A$-special subgroups of $G$ of degree $k$. Then $R_k=G^{(k)}$.

$(4)$ Let $H$ be an $A$-special subgroup of $G$. If $N$ is an $A$-invariant normal subgroup of $G$ then the image of $H$ in $G/N$ is an $A$-special subgroup of $G/N$.
\end{proposition}

\begin{theorem}\label{teoPS-CA2}
Let $A$ be an elementary abelian $q$-group of order $q^r$ with $r\geq 2$ acting on a finite $q'$-group $G$.  Let $p$ be a prime and   $P$ an $A$-invariant Sylow $p$-subgroup of $G^{(d)}$, for some  integer $d\geq 0$. Let $P_1,\ldots, P_t$ be  the subgroups of the form $P\cap H$ where $H$ ranges through $A$-special subgroups of $G$ of degree $d$. Then $P=P_1\cdots P_t$.
\end{theorem}

In order to deal  with the statements (1) of Theorems \ref{aa1} and \ref{aa2} we will need to use  the concept of $\gamma$-$A$-special subgroups of a group $G$, whose definition was also given in \cite{PS-CA}. They are analogues to  $A$-special subgroups defined above, and their definition is  more suitable to treat  situations involving the terms of the lower central series of a group. The definition is as follows. 
\begin{definition}\label{gamma-special}
Let $A$ be an elementary abelian $q$-group acting on a finite $q'$-group $G$. Let $A_1,\ldots,A_s$ be  the maximal subgroups of $A$ and $H$ a subgroup of $G$. 

$(1)$ We say that $H$ is a $\gamma$-$A$-special subgroup of $G$ of degree $1$ if and only if $H=C_G(A_i)$ for suitable $i\leq s$.

$(2)$ Suppose that $k\geq 2$ and the $\gamma$-$A$-special subgroups of $G$ of degree $k-1$ are defined. Then $H$ is a $\gamma$-$A$-special subgroup of $G$ of degree $k$ if and only if   there exist a $\gamma$-$A$-special subgroup $J$ of $G$ of degree $k-1$ such that  $H=[J, C_G(A_i)]\cap C_G(A_j)$ for suitable $i, j \leq s$.
\end{definition}

The next  results are similar to Proposition \ref{propPS-CA2} and Theorem \ref{teoPS-CA2} respectively, and   their proofs can be found in \cite{PS-CA}.

\begin{proposition}\label{propPS-CA}
Let $A$ be an elementary abelian $q$-group of order $q^r$ with $r\geq 2$ acting on a finite $q'$-group $G$ and $A_1,\ldots,A_s$ the maximal subgroups of $A$. Let $k\geq 1$ be an integer.

$(1)$ If $k\geq 2$, then every $\gamma$-$A$-special subgroup of $G$ of degree $k$ is contained in some $\gamma$-$A$-special subgroup of $G$ of degree $k-1$.

$(2)$ If $k\leq r-1$ and $H$ is a $\gamma$-$A$-special subgroup of $G$ of degree $k$, then $H\leq \gamma_{k}(C_G(B))$ for some subgroup $B\leq A$ such that $|A/B|\leq q^k.$ 

$(3)$ Let $R_k$ be the subgroup generated by all $\gamma$-$A$-special subgroups of $G$ of degree $k$. Then $R_k=\gamma_k(G)$.

$(4)$  Let $H$ be a $\gamma$-$A$-special subgroup of $G$. If $N$ is an $A$-invariant normal subgroup of $G$, then the image of $H$ in $G/N$ is a $\gamma$-$A$-special subgroup of $G/N$.
\end{proposition}

\begin{theorem}\label{teoPS-CA}
Let $A$ be an elementary abelian $q$-group of order $q^r$ with $r\geq 2$ acting on a finite $q'$-group $G$.  Let $p$ be a prime and   $P$ an $A$-invariant Sylow $p$-subgroup of $\gamma_{r-1}(G)$. Let $P_1,\ldots, P_t$ be all the subgroups of the form $P\cap H$ where $H$  ranges through $\gamma$-$A$-special subgroups of $G$ of degree $r-1$. Then $P=P_1\cdots P_t$.
\end{theorem}

Since Theorems \ref{bb1} and \ref{bb2} are (non-quantitative) profinite versions of Theorems \ref{aa1} and \ref{aa2} respectively, it is natural to expect that the main step of their proofs is to consider the case of pro-$p$ groups, which
will be treated  by using Lie methods.  For our purpose we  also need to  extend   the  concepts of $A$-special and $\gamma$-$A$-special subgroups to profinite groups.

Let $H,K$ be subgroups  of a profinite group $G$. Remind that we denote by $[H,K]$ the closed  subgroup of $G$ generated by all commutators of the form $[h,k]$, with $h\in H$ and $k\in K$. In  the same spirit of  what was done in \cite{PS-CA}  for the finite case, we can define the concept of $A$-special subgroups for a profinite group as follows.
\begin{definition}\label{def-special-profi}
Let $A$ be an elementary abelian $q$-group acting coprimely on a profinite group $G$. Let $A_1,\ldots,A_s$ be the maximal subgroups of $A$ and $H$ a subgroup of $G$.

$(1)$ We say that $H$ is an $A$-special subgroup of $G$ of degree $0$ if and only if  $H=C_G(A_i)$ for suitable $i\leq s$.

$(2)$ Suppose that $k\geq 1$ and the $A$-special subgroups of $G$ of degree $k-1$ are defined. Then $H$ is an $A$-special subgroup of $G$ of degree $k$ if and only if   there exist $A$-special subgroups $J_1, J_2$ of $G$ of degree $k-1$ such that  $H=[J_1, J_2]\cap C_G(A_i)$ for suitable $i\leq s$.
\end{definition}

Note that combining the definition above with a standard  inverse limit argument and the results obtained in \cite{PS-CA}, it is easy to show that $A$-special subgroups of a profinite group satisfy properties analogous to those listed in Proposition \ref{propPS-CA2}. Moreover a profinite version of  Theorem  \ref{teoPS-CA2} holds.

In order to deal with part (1) of Theorems \ref{bb1} and \ref{bb2} we need to introduce $\gamma$-$A$-special subgroups of a profinite group. This is done by slightly modifying Definition \ref{def-special-profi} in a similar way to what is stated in Definition \ref{gamma-special} for the finite case. As a consequence we obtain that analogous profinite versions of Proposition \ref{propPS-CA} and Theorem \ref{teoPS-CA}  can be established. We omit further details. 

 \section{Proof of Theorem \ref{aa1}} 
Our goal here is to prove Theorem \ref{aa1}. 
First of all we
%
need to establish the following result about associated Lie rings. 
\begin{proposition}\label{propaa1}
Let $G$ be a finite group satisfying the hypothesis of Theorem \ref{aa1}$(2)$. Suppose that there exists   an $A$-invariant $p$-subgroup $H$ of $G^{(d)}$, with $p$  a prime divisor of the order of $G^{(d)}$, such that $H=H_1\cdots H_t$, where each subgroup $H_i$ is contained in some $A$-special subgroup of $G$ of degree $d$ and $H$ is generated by a $\{q,r,t\}$-bounded number of elements. Then:

$(1)$ $L_p(H)$ is nilpotent of $\{n,p,q,r,t\}$-bounded class.

$(2)$ There exist positive integers $e,c$ depending only $n,q,r$ and $t$, such that $e\gamma_c(L(H))=0$.
\end{proposition}

We now deal with the proof of the second statement of Proposition \ref{propaa1}.
\begin{proof}
Let $L=L(H)$  be the Lie ring associated   with  the $p$-subgroup $H$ of $G^{(d)}$. Denote by  $V_1,\ldots, V_t$ the images of $H_1,\ldots, H_t$  in $H/{\gamma_2(H)}$. It follows that the Lie ring $L$ is generated by $V_1,\ldots,V_t$. 

Since $H$ is $A$-invariant, the group $A$ acts  on $L$ in the natural way. 
Let $A_1,\ldots, A_s$ be the distinct maximal subgroups of $A$. 
Let $W$ be an additive subgroup of $L$. We say that $W$ is a \textit{special  subspace} of weight $1$ if  and only if $W=V_j$ for some $j\leq t$ and say that $W$ is a special subspace of weight $\alpha\geq 2$ if $W=[W_1,W_{2}]\cap C_L(A_l)$, where $W_1, W_{2}$ are some special  subspaces of $L$ of weight $\alpha_{1}$ and $\alpha_{2}$ such that $\alpha_{1}+\alpha_{2}=\alpha$ and $A_l$ is some maximal subgroup of $A$ for a suitable $l\leq s$. 

We claim that every special subspace $W$ of $L$ corresponds to a subgroup of an $A$-special subgroup of $G$ of degree $d$. We argue by induction on the weight $\alpha$ of $W$. If $\alpha=1$, then $W=V_j$  and so $W$ corresponds to $H_j$ for some $j\leq t$. Assume that $\alpha\geq 2$ and write $W=[W_1,W_{2}]\cap C_L(A_l)$. By induction we know that $W_1, W_{2}$ correspond respectively to some $J_1, J_{2}$ which are subgroups of some $A$-special subgroups of $G$ degree $d$. 

Note also that $[W_1,W_{2}]$ is contained in the image of 
$[J_1,J_{2}]$. This implies that the special  subspace $W$ corresponds to a subgroup of $[J_1, J_{2}]\cap C_G(A_l)$ which, by Proposition \ref{propPS-CA2}(1), is contained in some $A$-special subgroup of $G$ of degree $d$, as claimed. 
Moreover it  follows from Proposition \ref{propPS-CA2}(2) that every  element of $W$ corresponds to some element of $C_G(a)^{(d)}$ for a suitable $a\in A^{\#}$.
Therefore,
\begin{equation}\label{Claim2}
\mbox{every element of}\ W\ \mbox{is ad-nilpotent}\ \mbox{of index at most}\ n,
\end{equation}
since all elements of $C_{G}(a)^{(d)}$ are $n$-Engel in $G$.

From the previous argument we deduce that $L=\left<V_1,\ldots, V_t\right>$  is generated by ad-nilpotent elements of index at most $n$, but we do not know whether every Lie commutator in these generators is again in some special subspace of $L$ and hence it is ad-nilpotent of bounded index. 
In order to overcome this problem we take  a $q$th primitive root of unity $\omega$ and put  $\overline{L}=L\otimes \mathbb{Z}[\omega]$. We regard $\overline{L}$ as a Lie ring and remark that there is a natural embedding of the ring $L$ into the ring $\overline{L}$.   In what follows we write $\overline{V}$ to denote $V\otimes \mathbb{Z}[\omega]$, for some subspace $V$ of $L$.

Let $W$ be a special  subspace of $L$. We claim that
\begin{equation}\label{Claim1}
\begin{aligned}
& \text{there exists a}\, \{n,q\}\text{-bounded number}\, u\, \text {such that every }\\
& \text{element}\,w\, \text{of}\, \overline{W}\,\text{is ad-nilpotent of index at most}\, u. \\ 	
\end{aligned}
\end{equation}
Indeed, chose $w\in \overline{W}$ and write
$$w=l_0+\omega l_{1}+\cdots+\omega^{q-2}l_{q-2},$$ for suitable  elements $l_0,\ldots,l_{q-2}$  of $W$ which  in particular correspond to some  elements $x_0,\ldots,x_{q-2}$ of $C_{G}(a)^{(d)}$, for a suitable $a\in A^{\#}$.  
Let denote  by $K=\langle l_{0}, \omega l_{1},\ldots, \omega^{q-2}l_{q-2}\rangle$ the subring of $\overline{L}$ generated by 
$ l_{0}, \omega l_{1},\ldots, \omega^{q-2}l_{q-2}$ and put $H_0=\langle x_0,\ldots,x_{q-2} \rangle$.  We will show that $K$ is nilpotent of $\{n,q\}$-bounded class. Note that $L(H,H_0)$ satisfies the linearized  $n$-Engel identity and, so, the same identity is also satisfied in $\overline{L(H,H_0)}$ which contains $K$.  Observe that 
a commutator in the elements $l_{0}, \omega l_{1},\ldots, \omega^{q-2}l_{q-2}$  is of the form $\omega^{\alpha}v$ for some $v\in W$, and so, by (\ref{Claim2}), it  is ad-nilpotent of index at most $n$. Hence, by Theorem \ref{Zelring}, $K$  is nilpotent of $\{n,q\}$-bounded class. Now Lemma \ref{lemmanovo} 
ensures that there exists a positive integer $u$, depending only on $n$ and $q$, such that $[\overline{L},_u K]=0$. Since  $w\in K$, we conclude that $w$ is ad-nilpotent in $\overline{L}$ with $\{n,q\}$-bounded index, as claimed in (\ref{Claim1}).

The group $A$ acts  on $\overline{L}$ in the natural way.
An element $x\in \overline{L}$ will be called a common ``eigenvector'' for $A$ if for any $a\in A^{\#}$ there exists a number $\lambda$ such that $x^a=\omega^\lambda x$. Since $(|A|,|G|)=1$ and $H$ can be generated by a $\{q,r,t\}$-bounded number of elements, we can choose elements $v_{1},\ldots,v_{\tau}$ in $\overline{V_{1}}\cup\cdots \cup\overline{V_{t}}$,  that generated   the Lie ring $\overline{L}$, where $\tau$ is a $\{q,r,t\}$-bounded number, and each $v_i$ is a common eigenvector for $A$  (see for example \cite[Lemma 4.1.1]{Khu1}). 

Let $v$ be any Lie commutator in $v_1,\ldots,v_\tau$. We wish to show that $v$ belongs to some $\overline{W}$, where $W$ is a special  subspace of $L$. We argue by induction on the weight of $v$. If $v$ has weight $1$ there is nothing to prove. Assume $v$ has weight at least $2$. Write $v=[w_1,w_2]$ for some $w_1\in \overline{W_1}$ and $w_2\in \overline{W_2}$, where $W_1, W_2$ are two special  subspace of $L$ of smaller weights. It is clear that $v$ belongs to $[\overline{W_1}, \overline{W_2}]$. Note that any commutator in common eigenvectors  is again a common eigenvector for $A$. Therefore $v$ is a common eigenvector and  it follows that there exists some maximal subgroup   $A_l$ of $A$ such that $v\in C_{\overline{L}}(A_l)$. Thus $v\in[\overline{W_1}, \overline{W_{2}} ]\cap C_{\overline{L}}(A_l)$. Hence $v$ lies in $\overline{W}$, where $W$ is the special  subspace of $L$ of the form $[{W_1}, W_{2}]\cap C_{{L}}(A_l)$ and so by (\ref{Claim1}), $v$ is ad-nilpotent of  index at most $u$. This proves that 
\begin{equation}
\mbox{any commutator in}\ v_1,\ldots,v_\tau\ \mbox{is ad-nilpotent of index at most }\, u.
\end{equation}

Note that  for any $a\in A^{\#}$, the centralizer $C_L(a)=L(H,C_H(a))$ sa\-tisfies the linearized version of the identity $[y,_n \delta_{2^{d}}(y_1,\ldots,y_{2^d}) ]\equiv0$, where  $\delta_{i}(y_{1},\ldots,y_{2^{i}})$ is given recursively by $$\delta_{0}(y_{1})=y_{1},\quad \delta_{i}(y_{1},\ldots,y_{2^{i}})=[\delta_{i-1}(y_{1},\ldots,y_{2^{i-1}}),\delta_{i-1}(y_{2^{i-1}+1},\ldots,y_{2^{i}})]$$ for any $i\geq 1$.  
The same identity also holds in $C_{\overline{L}}(a)=\overline{C_L(a)}$. Thus, by Theorem \ref{BazaRing}, there exist positive integers $e, c$ depending only on $n,q,r$ and $t$ such that $e\gamma_c(\overline{L})=0$. Since $L$ embeds into $\overline{L}$, we also have $e\gamma_c({L})=0$, as desired.
\end{proof}

Note that the proofs of the items (1) and (2) of  the previous proposition are very similar. As for the proof of item (1), we only observe that it can be obtained, with obvious changes, simply by replacing every appeal to Theorem \ref{BazaRing} in the proof of (2) by an appeal to Corollary \ref{ZelBazaField}.

We are now ready to embark on the proof of part (2) of Theorem \ref{aa1}. 
 \begin{proof}  By  Proposition \ref{propPS-CA2}(3)  we know that $G^{(d)}$ is generated by $A$-special subgroups of $G$ of degree $d$ and Proposition \ref{propPS-CA2}(2) tells us that any $A$-special subgroup of $G$ of degree $d$ is contained in $C_G(B)^{(d)}$ for some suitable nontrivial subgroup $B\leq A$ such that $|A/B|\leq q^{r-1}$. Thus  each $A$-special subgroup of $G$ of degree $d$ is contained in $C_G(a)^{(d)}$ for some suitable $a\in A^{\#}$. This implies that $G^{(d)}$ is generated by $n$-Engel elements.  Hence  by Baer's Theorem  \cite[III  6.14]{Huppert1}  we get that $G^{(d)}$ is nilpotent.  Then $G^{(d)}$ is a direct product of its Sylow subgroups.  

Let $\pi(G^{(d)})$ be the set of prime divisors of $|G^{(d)}|$.  
Choose now $p\in \pi(G^{(d)})$ and let $P$ be  the Sylow $p$-subgroup of $G^{(d)}$. By Theorem \ref{teoPS-CA2}, we have $P=P_1\cdots P_t$, where each $P_i$ is of the form $P\cap H$ for some $A$-special subgroup $H$ of $G$ of degree $d$. Combining this with Proposition \ref{propPS-CA2}(2) we see that each $P_i$ is contained in $C_G(a)^{(d)}$, for some $a\in A^{\#}$. Furthermore $t$ is a $\{q,r\}$-bounded number.
 
Choose arbitrarily $x,y\in P$. Let us write $x=x_{1}\cdots x_{t}$ and $y=y_{1}\cdots y_{t}$, where $x_{i}$ and $y_{i}$ belong to $P_{i}$. In what follows we will show that $\langle x,y\rangle$ is nilpotent of $\{n,q,r\}$-bounded class.  Let $Y$ be the subgroup  generated by the orbits $x_{i}^{A}$ and $y_{i}^{A}$, for $i=1,\ldots,t$. Note that $Y$ is generated by a $\{q,r\}$-bounded number of elements. Since the subgroup $\langle x,y\rangle$ is contained in  $Y$, it is enough to show that $Y$ is nilpotent of $\{n,q,r\}$-bounded class. 

Set $Y_{i}=P_{i}\cap Y$, for $i=1,\ldots,t$  and note that every $Y_{i}$ is a subgroup of $C_{G}(a)^{(d)}$ for a suitable $a\in A^{\#}$. Since $Y=\langle x_{i}^{A}, y_{i}^{A} \mid i=1, \ldots, t\rangle$ and every $P_{i}$ is an $A$-invariant subgroup  we have $Y=\langle Y_{1}, \ldots, Y_{t}\rangle$. By \cite[Lemma 2.1]{PS-CA}
 we see that $Y=Y_{1}\cdots Y_{t}$.  Moreover note  that $Y$ is generated by a $\{q,r\}$-bounded number of elements which are
$n$-Engel.

Now by  Proposition \ref{propaa1}(2) there exist integers $e,c$, that depend only on $n,q$ and $r$, such that $e\gamma_{c}(L(Y))=0$.  If $p$ is not a divisor of $e$, then we have $\gamma_{c}(L(Y)))=0$ and so $Y$ is nilpotent of class at most $c-1$. In that case  $Y$ is nilpotent of $\{n,q,r\}$-bounded class and, in particular, the same holds for $\langle x,y\rangle$.  Assume now that $p$ is a divisor of $e$. By Proposition \ref{propaa1}(1) we know that $L_{p}(Y)$ is nilpotent of $\{n,q,r\}$-bounded class. Now Theorem \ref{finiteL} tell us that $Y$ has a powerful characteristic subgroup $K$ of $\{n,q,r\}$-bounded index. It follows from \cite[6.1.8(ii), p.\ 164]{Rob} that  $K$ has   a $\{n,q,r\}$-bounded rank.  

Put $R=K^{e}$ and assume that $R\neq 1$. Note that, if $p\neq 2$, then we have
$$[R,R]\leq [K, K]^{e^{2}}\leq K^{pe^{2}}=R^{pe} ,$$
and if $p=2$, then we have
$$[R,R]\leq R^{4e}.$$

Since $e\gamma_{c}(L(R))=0$, we get that  $\gamma_{c}(R)^{e}\leq \gamma_{c+1}(R)$.  Taking into account that $R$ is powerful we obtain, if $p\neq 2$, that
$$\gamma_{c}(R)^{e}\leq \gamma_{c+1}(R)=[R',_{c-1} R]\leq [R^{pe},_{c-1} R]\leq \gamma_{c}(R)^{pe}$$
  and, if $p=2$,  that
$$\gamma_{c}(R)^{e}\leq \gamma_{c}(R)^{4e}.$$
Hence $\gamma_{c}(R)^{e}=1$.  Since $\gamma_{c}(R)$  is also powerful and generated by a $\{n,q,r\}$-bounded number of elements, we conclude that $\gamma_{c}(R)$ is of $\{n,q,r\}$-bounded order, since it is a product of  a $\{n,q,r\}$-bounded number of cyclic subgroups. It follows that $R$ has a $\{n,q,r\}$-bounded  derived length. Remind that $R=K^{e}$ and $K$ is a powerful $p$-group.  Thus,  $K$ has $\{n,q,r\}$-bounded derived length and  this implies that the derived length of $Y$  is $\{n, q, r\}$-bounded, as well.
Now 
\cite[Lemma 4.1]{shusa}  tell us  that $Y$ has $\{n,q,r\}$-bounded nilpotency class, and the same holds for  $\langle x,y\rangle$, as desired.  

From the  argument above we deduce that each Sylow $p$-subgroup of $G^{(d)}$ is $k$-Engel, for some $\{n,q,r\}$-bounded number $k$. The result follows.
\end{proof}

We conclude this section observing that the proof of part  (1) of Theorem \ref{aa1}  has a  very similar structure  to  that of part (2). We will omit details and describe only main steps that are somewhat different from those of part (2). More precisely the first step consists in  proving the following analogue of Proposition \ref{propaa1}.

\begin{proposition}
Let $G$ be a finite group satisfying the hypothesis of Theorem \ref{aa1}(1).  Suppose that there exists an $A$-invariant $p$-subgroup $H$ of $\gamma_{r-1}(G)$, with $p$  a prime divisor of the order of $\gamma_{r-1}(G)$, such that $H=H_1\cdots H_t$, where each subgroup $H_i$ is contained in some $\gamma$-$A$-special subgroup of $G$ of degree $r-1$ and $H$ is generated by a $\{q,r,t\}$-bounded number of elements. Then:

$(1)$ $L_p(H)$ is nilpotent of $\{n,p,q,r,t\}$-bounded class.

$(2)$ There exist positive integers $e,c$ depending only $n,q, r$ and $t$, such that $e\gamma_c(L(H))=0$.
\end{proposition}  

Next,  one can establish part (1) of Theorem \ref{aa1}  by replacing  every appeal to Theorem \ref{teoPS-CA2}  and Proposition  \ref{propPS-CA2}  in the proof of  item (2) by an appeal to Theorem \ref{teoPS-CA}  and Proposition \ref{propPS-CA},  respectively.

\section{Proof of Theorem \ref{aa2}}
In this section we are concerned with the proof of Theorem \ref{aa2}. In parallel to what we did in the previous section, we will focus our attention on the proof of  the statement (2) of Theorem \ref{aa2}.
 
First of all, we will require the following analogue of Proposition \ref{propaa1}. 
\begin{proposition}
\label{propaa2}
Let $G$ be a finite group satisfying the hypothesis of Theorem \ref{aa2}(2).  Suppose that there exists an $A$-invariant $p$-subgroup  $H$ of $G^{(d)}$, with $p$ a prime divisor of the order of $G^{(d)}$, such that $H=H_1\cdots H_t$, where   each subgroup $H_i$ is contained in some $A$-special subgroup of $G$ of degree $d$ and $H$ is generated by a $\{q,r,t\}$-bounded number of elements. Then:

$(1)$ $L_p(H)$ is nilpotent of $\{n,p,q,r, t\}$-bounded class.

$(2)$ There exist positive integers $e,c$ depending only $n,q,r$ and $t$, such that $e\gamma_c(L(H))=0$.
\end{proposition}
In what follows, we outline the proof of item (2) of Proposition \ref{propaa2}.  The proof of part (1) can be obtained with a similar argument. 
\begin{proof}
By hypothesis  we know that each subgroup $H_i$ for $i=1,\ldots,t$, is contained in some $A$-special subgroup of $G$ of degree $d$. Hence, Proposition \ref{propPS-CA2}(2) implies  that  each   $H_i$ is contained in $C_G(B)^{(d)}$, for some subgroup $B$ of $A$ such that $|A/B|\leq q^{2^{d}}$. Let $A_1,\ldots,A_s$ be the maximal subgroups of $A$. For any $A_j$ the intersection $B\cap A_j$ is not trivial. Thus, there exists $a\in A^{\#}$ such that the centralizer $C_G(A_j)$ is contained in $C_G(a)$ and $H_i$ is contained in $C_G(a)^{(d)}$. Since all elements of $C_G(a)^{(d)}$  are  $n$-Engel in $C_{G}(a)$, we deduce that 
\begin{equation}\label{aa2**}
\mbox{each element}\ h\in H_i\ \mbox{is $n$-Engel in}\ C_G(A_j), \mbox{for any}\  j\leq s.
\end{equation}

 We  now consider  $L=L(H)$  the Lie ring associated to  the $p$-subgroup $H$ of $G^{(d)}$.  In the same spirit of what we did in the proof of Proposition \ref{propaa1} we define special  subspaces  of $L$  for any weight and  observe that every element of a special subspace $W$ of $L$ corresponds to an element of  a subgroup of an $A$-special subgroup of $G$ of degree $d$. Since $2^d\leq r-2$ we have $L=\sum_{j\leq s}C_{L}(A_{j})$. Now  taking into account that every special subspace $W$ of $L$ is contained in some $L(H,H_{i})$  and that,  by (\ref{aa2**}), we have  $[C_{L}(A_{j}),_{n}L(H,H_{i})]=0$, we deduce that  any element of a special subspace $W$ is ad-nilpotent of index at most $n$ in $L$.

The rest of the proof consists in mimicking  the argument used in  the proof of Proposition \ref{propaa1}, with only obvious changes, so we omit the further details. 
\end{proof}

%
Now we are ready to   deal with the proof of  the part (2) of Theorem \ref{aa2}.

\begin{proof}
By well-known Zorn's Theorem \cite[Theorem III 6.3]{Huppert1} each $C_G(a)^{(d)}$ is nilpotent.  Furthermore  
\cite[Theorem 31]{PS-CA2} implies that  $G^{(d)}$ is nilpotent.  It follows that  $G^{(d)}$ is a direct product of its Sylow subgroups. 

Choose  $p\in \pi(G^{(d)})$ and let $P$ be a Sylow $p$-subgroup of $G^{(d)}$. 
By  Theorem \ref{teoPS-CA2}  we know that  $P=P_{1}\cdots P_{t}$, where $t$ is a $\{q,r\}$-bounded number and each subgroup $P_i$ is of the form $P\cap H$ where $H$ is some $A$-special subgroup of $G$ of degree $d$. 
 Moreover Proposition \ref{propPS-CA2}(2) tells us that  each $P_i\leq C_G(B)^{(d)}$, for some subgroup $B$ of $A$ such that $|A/B|\leq q^{2^d}$.

 Let $A_1,\ldots,A_s$ be the maximal subgroups of $A$. For any $A_j$ the intersection $B\cap A_j$ is not trivial. Thus, there exists $a\in A^{\#}$ such that the centralizer $C_G(A_j)$ is contained in $C_G(a)$ and $P_i$ is contained in $C_G(a)^{(d)}$. Thus, 
\begin{equation}\label{condEngel}
\mbox{each element}\ l\in P_i\ \mbox{is $n$-Engel in}\ C_G(A_j), \mbox{for any}\  j\leq s.
\end{equation}
Choose arbitrarily  $x, y\in P$. We will show that the subgroup $\left< x, y\right>$  is nilpotent of $\{n, q, r\}$-bounded class. 
Following an argument similar to that used in the proof of Theorem \ref{aa1}(2) we write $x=x_1\cdots x_t$ and $y=y_1\cdots y_t$, where each  $x_i$ and $y_i$ belongs to $P_i$, for $i=1,\ldots, t$, and  want to show that  the subgroup $Y=\left< x_i^A, y_i^A|i=1,\ldots,t\right>$ is nilpotent of $\{n,q,r\}$-bound class.  Appealing to Proposition  \ref{propaa2} and Theorem \ref{finiteL}  we find  that $Y$ has a powerful characteristic subgroup $N$ of $\{n,q,r\}$-bounded index and, by  \cite[6.1.8(ii), p.\ 164]{Rob}, of   $\{n,q,r\}$-bounded rank as well.

%

Thus each  $C_N(a)^{(d)}$ is an $n$-Engel subgroup and can be generated by a $\{n,q,r\}$-bounded number of elements. Zelmanov noted in \cite{Z2} that the nilpotency class of a finite $n$-Engel group is bounded in terms of $n$ and the number of generators of that group. We conclude that each $C_N(a)^{(d)}$ is nilpotent of $\{n,q,r\}$-bounded class. Now 
\cite[Theorem 31]{PS-CA2} tell us  that $N^{(d)}$ is nilpotent of $\{n,q,r\}$-bounded class. This implies that  $Y$ has $\{n, q, r\}$-bounded derived length $l$, say. 

 In  view of
 \cite[Lemma 4.1]{shusa}, it is enough to see that there exists an  $\{n,q,r\}$-bounded number $u$ such that each generator of $Y$ is an  $u$-Engel element in $Y$. Indeed, let $M=Y^{(l-1)}$ be  the last  nontrivial term of derived series of $Y$.  By induction on $l$ we see that $Y/M$ is of $\{n, q, r\}$-bounded  nilpotency class, say $c_1$. For any generator  $x$  of $Y$  we have
$[Y,\underbrace{x,\ldots,x}_{c_1}]\subseteq M.$ Since $M$ is abelian and $A$-invariant we can write $M=\sum_{j\leq s}C_M(A_j).$ By (\ref{condEngel}) we have $[M,\underbrace{x,\ldots,x}_{n}]=1$. Thus, $[Y,\underbrace{x,\ldots,x}_{c_1+n}]=1$.   This implies that $\langle x,y\rangle$ is  nilpotent of $\{n,q,r\}$-bound class, as desired.

By the argument above  we get that each Sylow $p$-subgroup $P$ of $G^{(d)}$ is $k$-Engel for some $\{n,q,r\}$-bounded number $k$. The result follows. 
\end{proof}

We finish by noting that the proof of item (1) of Theorem \ref{aa2} can be obtained by replacing  every appeal to  
\cite[Theorem 31]{PS-CA2}, Theorem \ref{teoPS-CA2}, Proposition  \ref{propPS-CA2} and Proposition \ref{propaa2}  in the proof of (2) by an appeal to \cite[Theorem 41]{PS-CA2}, Theorem \ref{teoPS-CA}, Proposition \ref{propPS-CA} and the  following analogue of Proposition \ref{propaa2}, respectively.
\begin{proposition}
\label{propaa2gamma}
Let $G$ be a finite group satisfying the hypothesis of Theorem \ref{aa2}(1). Suppose that there exists an $A$-invariant $p$-subgroup  $H$ of $\gamma_{r-2}(G)$,  with $p$ a prime divisor of the order of $\gamma_{r-2}(G)$, such that $H=H_1\cdots H_t$, where  each subgroup $H_i$ is contained in some $\gamma$-$A$-special subgroup of $G$ of degree $r-2$ and $H$ is generated by a $\{q,r,t\}$-bounded number of elements. Then:

$(1)$ $L_p(H)$ is nilpotent of $\{n,p,q,r, t\}$-bounded class.

$(2)$ There exist positive integers $e,c$ depending only $n,q,r$ and $t$, such that $e\gamma_c(L(H))=0$.
\end{proposition}

\section{Results on profinite groups}
In this section we deal with  the proofs of Theorems \ref{bb1} and \ref{bb2} which are  profinite non-quantitative analogues of  Theorems \ref{aa1} and \ref{aa2} respectively.  

Let $w=w(x_1,x_2,\dots,x_k)$ be a group-word. Let $H$ be a subgroup of a group $G$ and $g_1,g_2,\dots,g_k\in G$. We say that the law $w\equiv1$ is satisfied on the cosets $g_1H,g_2H,\dots,g_kH$ if $w(g_1h_1,g_2h_2,\dots,g_kh_k)=1$ for all $h_1,h_2,\dots,h_k\in H$.  Let us start  with a lemma.

\begin{lemma}\label{lpi}
Assume that a finite group $A$ acts coprimely on a profinite group $G$. Then for each prime $p$ the following holds:

$(1)$ If, for some integer $k$, all elements in $\gamma_k(C_G(A))$ are Engel in $C_G(A)$,  then $L_p(G)$ satisfies a multilinear Lie polynomial identity.

$(2)$ If, for some integer $k$,  all elements in the $k$th derived group of $C_G(A)$ are Engel in $C_G(A)$, then $L_p(G)$ satisfies a multilinear Lie polynomial identity.
\end{lemma}
The proofs of the items (1) and (2) of the lemma are similar, so we give a detailed proof of the second statement. 
 \begin{proof} Let $L=L_p(G)$.
In view of Theorem \ref{blz} it is sufficient to show that $C_L(A)$ satisfies a polynomial identity. We know that $C_L(A)$ is isomorphic to the Lie algebra associated with the central series of $C_G(A)$ obtained by intersecting $C_G(A)$ with the $p$-dimension central series of $G$. 

Let $T=\underbrace{C_G(A)\times\cdots\times C_G(A)}_{2^k+1}$. For each integer $i$ we define the set $$S_i=\{(t,t_1,\ldots,t_{2^k})\in T :[t,_i\delta_k(t_1,\ldots,t_{2^k})]=1\}.$$
Since the sets $S_i$ are closed in $T$ and their union coincides with $T$, by the Baire category theorem \cite[p.\ 200]{Baire} at least one of these sets has a non-empty interior. Therefore, we can find an open subgroup $H$ in $C_G(A)$, elements $g,g_1,\ldots,g_{2^k}\in C_G(A)$ and an integer $n$ such that the identity $[x,_n\delta_k(y_1,\ldots,y_{2^k})]\equiv 1$ is satisfied on the cosets $gH,g_1H,\ldots,g_{2^k}H$. Thus, the Wilson-Zelmanov result \cite[Theorem 1]{WZ} tell us that $C_L(A)$ satisfies a  polynomial identity. 
\end{proof}

We will also require the following profinite version of  \cite[Lemma 2.1]{PS-CA}. The proof is straightforward so we  do not give details.
\begin{lemma}\label{prole2.1}
Suppose that a pronilpotent group $G$ is generated by subgroups $G_1,\ldots,G_t$ such that $\gamma_i(G)=\left<\gamma_i(G)\cap G_j:1\leq j\leq t\right>$ for all $i\geq 1.$ Then $G=G_1\cdots G_t$.
\end{lemma}

As usual, for a profinite group $G$ we denote by $\pi(G)$ the set of prime divisors of the orders of finite continuous homomorphic images of $G$.  We say that $G$ is a $\pi$-group if $\pi(G)$ is contained in $\pi$ and $G$ is a $\pi'$-group if $\pi(G)\cap\pi=\emptyset$. If $\pi$ is a set of primes, we denote by $O_{\pi}(G)$ the maximal normal $\pi$-subgroup of $G$ and by $O_{\pi'}(G)$ the maximal normal $\pi'$-subgroup.

We are ready to embark on the proof of Theorem \ref{bb1}. 
\begin{proof}[{ Proof of Theorem \ref{bb1}(2)}]
Let $\mathcal{S}$ be the subset of all $A$-invariant open normal subgroups of $G$.  For any $N\in \mathcal{S}$, set $Q=G/N$. Observe  that  $C_{Q}(a)^{(d)}=(C_G(a))^{(d)}N/N$, for any $a\in A^{\#}$.
By hypothesis each  $x\in C_G(a)^{(d)}$ is Engel in $G$ and  this implies that   all elements of  $C_{Q}(a)^{(d)}$ are Engel in $Q$. Hence Theorem \ref{aa1}(2) tell us  that $Q^{(d)}$ is Engel. By Zorn's Theorem \cite[Theorem III 6.3]{Huppert1}  $Q^{(d)}$ is nilpotent. Since  $G^{(d)}\cong\underleftarrow{\lim}_{N\in \mathcal{S}}\, (G/N)^{(d)}$, we get that $G^{(d)}$ is  a pronilpotent group. Thus, $G^{(d)}$ is the Cartesian product of its Sylow subgroups.

Choose $a\in A^{\#}$. For each positive integer $i$ we set 
$$S_i=\{(x,y)\in G^{(d)}\times C_G(a)^{(d)}:[x,_iy]=1\}.$$ Since the sets $S_i$ are closed in $G^{(d)}\times C_G(a)^{(d)}$ and their union coincides with $G^{(d)}\times C_G(a)^{(d)}$, by the Baire category theorem  at least one $S_i$ has a non-empty interior. Therefore we can find an integer $n$, an open subgroup $K$ in $G^{(d)}$, elements $u\in G^{(d)}$ and $v\in C_G(a)^{(d)}$ such that 
$$[ul,_nvk]=1, \ \ \ \mbox{for any}\ l\in K\ \mbox{and any}\ k\in K\cap C_G(a)^{(d)}.$$
Let $[G^{(d)}:K]=m$ and let $\pi_1=\pi(m)$ be the set of primes dividing $m$. Denote $O_{\pi_1'}(G^{(d)})$ by $T$.
Since $T$ is isomorphic to the image of $K$ in $G^{(d)}/O_{\pi_1}(G^{(d)})$, it is easy to see that $[T,_nx]=1$ for all $x\in C_T(a)^{(d)}$. Thus, each element of $C_T(a)^{(d)}$ is $n$-Engel in $T$.

The open subgroup $K$, the set $\pi_1$ and the integer $n$ depend only on the choice of $a\in A^{\#}$, so strictly speaking they should be denoted by $K_a, \pi_a$ and $n_a$, respectively. We choose such $K_a, \pi_a$ and $n_a$ for any $a\in A^{\#}$. Set $\pi=\cup_{a\in A^{\#}}\pi_a$, $n=\mathrm{max}\{n_a:a\in A^{\#}\}$ and $R=O_{\pi'}(G^{(d)})$. The choice of the set $\pi$ guarantees that for each $a\in A^{\#}$ every element of $C_R(a)^{(d)}$ is $n$-Engel in $R$. Using a routine inverse limit argument we deduce from Theorem 
\ref{aa1}(2) that $R$ is $n_1$-Engel for some suitable integer $n_1$. By \cite[Theorem 5]{WZ} $R$ is locally nilpotent. Let $p_1,p_2,\ldots, p_r$ be the finitely many primes in $\pi$ and $S_1,\ldots,S_r$ be the corresponding Sylow subgroups of $G^{(d)}$. Then $G^{(d)}=S_1\times \cdots \times S_r\times R$ and therefore it is enough to show that each subgroup $S_i$ is locally nilpotent.

Let $P$ be such a $p$-Sylow subgroup of $G^{(d)}$, for some $p\in \pi$. By the profinite version of Theorem \ref{teoPS-CA2} we have $P=P_1\cdots P_t$, where any $P_j=P\cap H$ and $H$ is an $A$-special subgroup of $G$ of degree $d$. Since $2^d\leq r-1$ by the profinite version of Theorem \ref{propPS-CA2}(2) each subgroup $P_j\subseteq C_G(a)^{(d)}$, for some suitable $a\in A^{\#}$. 

Choose arbitrary elements $x_1,\ldots,x_m$ in $P$. For $i=1,2,\ldots,m$  we write $x_i=x_{i1}\cdots x_{it}$, where $x_{ij}$ belongs to $P_j$ for $j=1,2,\ldots,t$.	Let $Y$ be the  subgroup generated by the orbits $x^A_{ij}$. Put $Y_j=Y\cap P_j$. Since $Y$ is generated by the orbits $x^A_{ij}$ and every $P_j$ is an $A$-invariant subgroup we deduce that $Y$ is generated by $Y_1,\ldots, Y_t$ and, by  Lemma \ref{prole2.1}, we have  $Y=Y_1\cdots Y_t$.

For our purpose it is sufficient to show that $Y$ is nilpotent.
We denote by  $D_j=D_j(Y)$ the terms of the $p$-dimension  central series of $Y$. Let $L=L_p(Y)$  be the Lie algebra  generated by $D_1/D_2$ associated with the pro-$p$ group $Y$. Observe that $D_2$ coincides with $\Phi(Y)$ and  denote by  $V_1,\ldots, V_t$ the images of $Y_1,\ldots, Y_t$  in $Y/\Phi(Y)$. It follows that the Lie algebra $L$ is generated by $V_1,\ldots,V_t$. 

Since $Y$ is $A$-invariant, the group $A$ acts  on $L$ in the natural way. 
Let $A_1,\ldots, A_s$ be the distinct maximal subgroups of $A$. 
In the same spirit of what was done in the proof of  Proposition \ref{propaa1} we can define special subspaces of $L$ of any weight and 
%
show that every special subspace $W$ of $L$ corresponds to a subgroup of an $A$-special subgroup of $G$ of degree $d$.
Moreover it  follows from the profinite version of  Proposition \ref{propPS-CA2}(2) that every  element of $W$ corresponds to some element of $C_G(a)^{(d)}$ for a suitable $a\in A^{\#}$ and so, by Lemma \ref{leWZ}, it is ad-nilpotent on $L$, since all elements of $C_{G}(a)^{(d)}$ are Engel in $G$.

From the previous argument we deduce that $L=\left<V_1,\ldots, V_t\right>$  is generated by ad-nilpotent elements. As in the proof of Proposition \ref{propaa1} we extend the ground field $\mathbb{F}_p$ by a primitive $q$th root of unity $\omega$. Put $\overline{L}=L\otimes \mathbb{F}_p[\omega]$ and  identify, as usual, $L$ with the $\mathbb{F}_p$-subalgebra $L\otimes 1$ of $\overline{L}$.
In what follows we write $\overline{X}$ to denote $X\otimes \mathbb{F}_p[\omega]$, for some subspace $X$ of $L$. 
%
Let $W$ be a special  subspace of $L$. We claim that \begin{equation}\label{Claim1profinite}
\mbox{every element}\ w\ \mbox{of}\ \overline{W}\ \mbox{is ad-nilpotent in}\ \overline{L}.
\end{equation}
Indeed, choose $w\in \overline{W}$ and write
$$w=l_0+\omega l_{1}+\cdots+\omega^{q-2}l_{q-2},$$ for suitable  elements $l_0,\ldots,l_{q-2}$  of $W$ that,  in particular, correspond to some  elements $x_0,\ldots,x_{q-2}$ of $C_{G}(a)^{(d)}$, for a suitable $a\in A^{\#}$.  
Let denote  by $K_0=\langle l_{0}, \omega l_{1},\ldots, \omega^{q-2}l_{q-2}\rangle$ the subalgebra of $\overline{L}$ generated by $ l_{0}, \omega l_{1},\ldots, \omega^{q-2}l_{q-2}$.
Using a analogous  argument to that used in the proof of Proposition \ref{propaa1} we first apply  Theorem \ref{Z1992} to  show that  $K_0$  is nilpotent and, later, by appealing to  Lemma \ref{lemmanovo}, we conclude that $w$ is ad-nilpotent in $\overline{L}$, as claimed in (\ref{Claim1profinite}). 


The group $A$ acts  on $\overline{L}$ in the natural way and now the ground field is a splitting field for $A$. Since $Y$ is finitely generated, we can choose finitely many elements $v_{1},v_2,\ldots$ in $\overline{V_{1}}\cup\cdots \cup\overline{V_{t}}$,  that generate   the Lie algebra $\overline{L}$, and each $v_i$ is a common eigenvector for $A$.
Let $v$ be any Lie commutator in generators $v_1,v_2\ldots$ of $\overline{L}$. Mimicking what we did in the proof of Proposition \ref{propaa1} and arguing by induction on the weight of $v$ we can show that $v$ belongs to some $\overline{W}$, where $W$ is a special  subspace of $L$. Thus, by (\ref{Claim1profinite}), $v$ is ad-nilpotent.
%
This proves that
\begin{equation}
\mbox{any commutator in}\ v_1,v_2\ldots\ \mbox{is ad-nilpotent in } \overline{L}.
\end{equation}

Furthermore, it follows from Lemma \ref{lpi}(2) that $L$ satisfies a multilinear Lie polynomial identity. The multilinear identity is also satisfied in $\overline{L}$  and from Theorem \ref{Z1992} we deduce that $\overline{L}$ is nilpotent. Since $L$ embeds into $\overline{L}$, we get that $L$ is nilpotent as well.

According to Lazard \cite{L} the nilpotency of $L$ is equivalent to $Y$ being $p$-adic analytic (for details see  \cite[A.1 in Appendice and  Sections 3.1 and 3.4 in Ch.\ III]{L} or  \cite[1.(k) and 1.(o) in Interlude A]{GA}). By \cite[7.19 Theorem]{GA} $Y$ admits a faithful linear representation over the field of $p$-adic numbers. A result of Gruenberg \cite[Theorem 0]{G} says that in a linear group the Hirsch-Plotkin radical coincides with the set of Engel elements. 
Since $Y=Y_1\cdots Y_t$ is finitely generated and 
each $Y_j$ is contained in $C_G(a)^{(d)}$ for some suitable $a\in A^{\#}$, we deduce that $Y$ is nilpotent.
 The proof is complete.
 \end{proof}
 
The proof of part (1)  of Theorem \ref{bb1}  is analogous to that of item (2) and can be  obtained by replacing  every appeal to Theorem \ref{aa1}(2), Theorem \ref{teoPS-CA2}, Proposition  \ref{propPS-CA2} and Lemma 
\ref{lpi}(2) in the proof of (2) by an appeal to Theorem \ref{aa1}(1), Theorem \ref{teoPS-CA}, Proposition \ref{propPS-CA} and Lemma \ref{lpi}(1), respectively. Therefore we  omit the further details.

In what follows we give an outline of the proof of Theorem \ref{bb2}.
\begin{proof}[Proof of Theorem \ref{bb2}(2)]
With an argument similar to that  used in the proof of  Theorem \ref{bb1}(2) and appealing to Theorem \ref{aa2}(2) it easy to show that $G^{(d)}$ is pronilpotent and so it is the Cartesian product of its Sylow subgroups.

Choose $a\in A^{\#}$. Since $C_G(a)^{(d)}$ is Engel,  \cite[Theorem 5]{WZ} guarantees that $C_G(a)^{(d)}$ is locally nilpotent. By \cite[Lemma 2.5]{PS-CA3} there exists a positive integer $n$, elements $u,v\in C_G(a)^{(d)}$ and an open subgroup $H\leq C_G(a)^{(d)}$ such that the law  $[x,_n y]\equiv 1$ is satisfied on the cosets $uH, vH$. Let 
$[C_G(a)^{(d)}:H]=m$ and $\pi_1=\pi(m)$. Denote $O_{\pi'_1}(C_G(a)^{(d)})$ by $T$. Thus, $T$ satisfies the law $[x,_n y]\equiv 1$, that is, $T$ is $n$-Engel. 
%
%
By the result of Burns and Medvedev \cite{BM} the subgroup $T$ has a nilpotent normal subgroup $U$ such that $T/U$ has finite exponent, say $e$. Set $\pi_2=\pi(e)$. The sets $\pi_1$ and $\pi_2$ depend on the choice of $a\in A^{\#}$, so strictly speaking they should be denoted by $\pi_1(a)$ and $\pi_2(a)$. For each such choice let $\pi_a=\pi_1(a)\cup\pi_2(a)$. We repeat this argument for every $a\in A^{\#}$. Set 
$\pi=\cup_{a\in A^{\#}}\pi_a$ and $R=O_{\pi'}(G^{(d)})$. Since all sets $\pi_1(a)$ and $\pi_2(a)$ are finite, so is $\pi$. 
Let $p_1,\ldots,p_r$ be the finitely many primes in $\pi$ and $S_1,\ldots,S_r$ be the corresponding Sylow subgroups of $G^{(d)}$. Then $G^{(d)}=S_1\times\cdots\times S_r\times R$.

The choice of the set $\pi$ guarantees that $C_R(a)^{(d)}$ is nilpotent for every $a\in A^{\#}$. 
Using the routine inverse limit argument we deduce from  
 \cite[Theorem 31]{PS-CA2} that $R^{(d)}$ is nilpotent. Thus $R$ is solvable. We claim that $R$ is an Engel group. Indeed, combining  the profinite version of  Proposition \ref{propPS-CA2}(3) with Lemma \ref{prole2.1} we obtain that $R=R_1\cdots R_t$,
where $R_i=R\cap H$ and $H$ is some $A$-special subgroup of $G$ of degree $d$.
Choose arbitrarily $x,y\in R$. Therefore it is sufficient to prove that $\left<x,y\right>$ is nilpotent. Note that we can write $x=x_1\cdots x_t$ and $y=y_1\cdots y_t$, where each $x_i$ and $y_i$ belongs to $R_i$, for $i=1,\ldots,t$. Consider $Y=\left<x^A_i,y_i^A:i\leq t\right>$ and set $K$ be the abstract  subgroup generated by the elements $x^A_i$ and $y_i^A$.  Since $K$ is a dense subgroup of $Y$, in order to prove that $Y$ is nilpotent is enough to prove that $K$ is nilpotent. 

 By construction and since   $2^d\leq r-2$, there exists $a\in A^{\#}$ such that  the centralizer $C_G(A_j)$ is contained in $C_G(a)$ and each subgroup $R_i$ is contained in $C_G(a)^{(d)}$. Thus
\begin{equation}\label{eq1.9}
\mbox{each element}\ x\in R_i\ \mbox{is Engel in}\ C_G(A_j),\ \mbox{for any}\ j\leq s.
\end{equation}
In the same spirit of what was done in the proof of Theorem \ref{aa2}(2), it is possible to show, by induction on  the derived length of $K$, that all generators of $K$ are Engel elements in $K$. Now by a well-known result of Gruenberg \cite[12.3.3]{Rob} we conclude that $K$ is nilpotent and, so, $Y$ is nilpotent, as well. 
In particular, we deduce that $R$ is Engel, as desired. Thus, by \cite[Theorem 5]{WZ} $R$ is locally nilpotent. 
Since $G^{(d)}=S_1\times\cdots\times S_r\times R$, for our purpose it is sufficient to prove that each subgroup $S_i$ is locally nilpotent, for $i=1,\ldots,r$.

Let $P$ be such a $p$-Sylow subgroup of $G^{(d)}$, for some $p\in \pi$. By the profinite version of Theorem \ref{teoPS-CA2} we have $P=P_1\cdots P_t$, where any $P_j=P\cap H$ and $H$ is an $A$-special subgroup of $G$ with degree $d$.


Choose  arbitrary elements $x_1,\ldots,x_m$ in $P$. Let us write $x_i=x_{i1}\cdots x_{it}$, for $i=1,\ldots,m$, where  each $x_{ij}$ belongs to $P_j$ and so  to $C_G(a)^{(d)}$ for a suitable $a\in A^{\#}$. Let $X$ be the  subgroup generated by the orbits $x^A_{ij}$ and let $L=L_p(X)$. By the assumptions we have  $L=\sum_{j\leq s}C_L(A_j)$ and using the Lie theoretical machinery it is possible to prove that $L$ is nilpotent.

%
%
%
%

According to Lazard \cite{L} the nilpotency of $L$ is equivalent to $X$ being $p$-adic analytic. The Lubotzky and Mann theory \cite{luma} now tells us that $X$ is of finite rank, that is, all closed subgroups of $X$ are finitely generated. In particular, we conclude that $C_X(a)^{(d)}$ is finitely generated for every $a\in A^{\#}$. It follows from \cite[Theorem 5]{WZ} that $C_X(a)^{(d)}$ is nilpotent. By  the profinite quantitative version of 
\cite[Theorem 31]{PS-CA2} $X^{(d)}$ is nilpotent. Then $X$ is soluble.  
Finally by mimicking what we did above for $Y$ we can prove that $X$ is nilpotent. This concludes the proof.
\end{proof}
The proofs of part (1) and (2) of Theorem \ref{bb2} are very similar. We conclude  noting that the proof of item (1) of Theorem \ref{bb2} can be obtained by replacing  every appeal to Theorem \ref{aa2}(2), Theorem \ref{teoPS-CA2}, Proposition  \ref{propPS-CA2} and Lemma \ref{lpi}(2) in the proof of (2) by an appeal to Theorem \ref{aa2}(1) Theorem \ref{teoPS-CA}, Proposition \ref{propPS-CA} and Lemma \ref{lpi}(1), respectively. Therefore we will omit further details.

\end{document}